\documentclass[a4paper]{amsart}
\usepackage[utf8]{inputenc}
\usepackage[T1]{fontenc}
\usepackage{lmodern, enumerate}
\usepackage{amssymb,amsxtra}
\usepackage[all]{xy}
\usepackage{nicefrac,mathtools,enumitem}
\usepackage{microtype}
\usepackage[pdftitle={Inverse Semigroup Actions as Groupoid Actions},
  pdfauthor={Alcides Buss, Ruy Exel, Ralf Meyer},
  pdfsubject={Mathematics}
]{hyperref}
\usepackage[lite]{amsrefs}
\newcommand*{\MRref}[2]{ \href{http://www.ams.org/mathscinet-getitem?mr=#1}{MR \textbf{#1}}}
\newcommand*{\arxiv}[1]{\href{http://www.arxiv.org/abs/#1}{arXiv: #1}}

\renewcommand{\PrintDOI}[1]{\href{http://dx.doi.org/\detokenize{#1}}{doi: \detokenize{#1}}}

\BibSpec{book}{%
    +{}  {\PrintPrimary}                {transition}
    +{,} { \textit}                     {title}
    +{.} { }                            {part}
    +{:} { \textit}                     {subtitle}
    +{,} { \PrintEdition}               {edition}
    +{}  { \PrintEditorsB}              {editor}
    +{,} { \PrintTranslatorsC}          {translator}
    +{,} { \PrintContributions}         {contribution}
    +{,} { }                            {series}
    +{,} { \voltext}                    {volume}
    +{,} { }                            {publisher}
    +{,} { }                            {organization}
    +{,} { }                            {address}
    +{,} { \PrintDateB}                 {date}
    +{,} { }                            {status}
    +{}  { \parenthesize}               {language}
    +{}  { \PrintTranslation}           {translation}
    +{;} { \PrintReprint}               {reprint}
    +{.} { }                            {note}
    +{.} {}                             {transition}
    +{} { \PrintDOI}                   {doi}
    +{}  {\SentenceSpace \PrintReviews} {review}
}

\numberwithin{equation}{section}
\theoremstyle{plain}
\newtheorem{theorem}[equation]{Theorem}
\newtheorem{lemma}[equation]{Lemma}
\newtheorem{proposition}[equation]{Proposition}

\theoremstyle{definition}
\newtheorem{definition}[equation]{Definition}
\theoremstyle{remark}
\newtheorem{remark}[equation]{Remark}
\newtheorem{example}[equation]{Example}

\DeclareMathOperator{\GrdS}{\textsf{\textup{Gr}}}
\DeclareMathOperator{\Bisec}{\textsf{\textup{Bis}}}
\DeclareMathOperator{\Prim}{Prim}

\newcommand*{\Zak}{\curvearrowright}

\newcommand*{\nb}{\nobreakdash}

\newcommand*{\Open}{\mathbb O}
\newcommand*{\OpenG}{\mathcal O} 

\newcommand*{\Top}{\textsf {\textup{Top}}}
\newcommand*{\Forget}{\textsf {\textup{Forget}}}
\newcommand*{\Topop}{\textsf {\textup{Top}}^\textup{op}}
\newcommand*{\Semil}{\textsf {\textup{Sela}}}
\newcommand*{\Invse}{\textsf {\textup{Mon}}^*}
\newcommand*{\Grpds}{\textsf {\textup{Grd}}}

\newcommand*{\Cst}{\textup C^*}
\newcommand*{\Grd}{\mathcal G}
\newcommand*{\GrdH}{\mathcal H}

\newcommand*{\Id}{\textup{Id}}

\newcommand*{\Cat}{\mathcal C}

\newcommand*{\defeq}{\mathrel{\vcentcolon=}}

\newcommand*{\cl}[1]{\overline{#1}}

\newcommand*{\germ}[2]{[#1, #2]} 
\newcommand*{\s}{\textup s}
\renewcommand*{\r}{\textup r}

\begin{document}
\title{Inverse semigroup actions as groupoid actions}

\author{Alcides Buss}
\email{alcides@mtm.ufsc.br}
\address{Departamento de Matem\'atica\\
 Universidade Federal de Santa Catarina\\
 88.040-900 Florian\'opolis-SC\\
 Brazil}

\author{Ruy Exel}
\email{exel@mtm.ufsc.br}
\address{Departamento de Matem\'atica\\
 Universidade Federal de Santa Catarina\\
 88.040-900 Florian\'opolis-SC\\
 Brazil}

\author{Ralf Meyer}
\email{rmeyer2@uni-goettingen.de}
\address{Mathematisches Institut and Courant Research Centre ``Higher Order Structures''\\
 Georg-August-Universit\"at G\"ottingen\\
 Bunsenstra\ss e 3--5\\
 37073 G\"ottingen\\
 Germany}

\begin{abstract}
  To an inverse semigroup, we associate an étale groupoid such that its actions on topological spaces are equivalent to actions of the inverse semigroup.  Both the object and the arrow space of this groupoid are non-Hausdorff.  We show that this construction provides an adjoint functor to the functor that maps a groupoid to its inverse semigroup of bisections, where we turn étale groupoids into a category using algebraic morphisms.  We also discuss how to recover a groupoid from this inverse semigroup.
\end{abstract}
\thanks{Supported by the German Research Foundation (Deutsche Forschungsgemeinschaft (DFG)) through the Institutional Strategy of the University of G\"ottingen and the grant ``Actions of \(2\)\nb-groupoids on C*\nb-algebras.''  Supported by CNPq and MATH-AmSud project U11MATH-05.}
\maketitle

\section{Introduction}
\label{sec:introduction}

Étale topological groupoids are closely related to actions of inverse semigroups on topological spaces by partial homeomorphisms (see \cites{Exel:Inverse_combinatorial, Matsnev-Resende:Etale_Groupoids, Paterson:Groupoids}).  This relationship is used, in particular, to study actions of étale topological groupoids on \(\Cst\)\nb-algebras and their crossed products.  In order to construct an étale topological groupoid out of an inverse semigroup, we first need it to act on some topological space.  We propose and study a particularly natural action of a given inverse semigroup.

Namely, given an inverse semigroup~\(S\) with idempotent semilattice \(E\subseteq S\), we consider the canonical action of~\(S\) on the character space~\(\hat{E}\) of~\(E\), equipped with a certain canonical \emph{non-Hausdorff} topology.  With a different, Hausdorff topology, this action of~\(S\) has already been studied in \cites{Exel:Inverse_combinatorial,Paterson:Groupoids}.  Our non-Hausdorff topology has the following crucial feature: if \(\GrdS(S)\) is the étale groupoid of germs for the action of~\(S\) on~\(\hat{E}\), then the category of actions of~\(S\) on topological spaces is equivalent to the category of actions of \(\GrdS(S)\) on topological spaces.  The action on~\(\hat{E}\) is the unique action of~\(S\) with this property, and it is a terminal object in the category of actions of~\(S\) on topological spaces.

This construction also sheds light on how to turn groupoids into a category.  The map \(S\mapsto \GrdS(S)\) is functorial, and left adjoint to the functor~\(\Bisec\) that maps an étale topological groupoid~\(\Grd\) to its inverse semigroup of bisections \(\Bisec(\Grd)\), \emph{provided} we turn étale groupoids into a category in an unusual way, using a notion of morphism due to Zakrzewski (see~\cite{Buneci-Stachura:Morphisms_groupoids}): an algebraic morphism from~\(\Grd\) to~\(\GrdH\), denoted \(\Grd\Zak \GrdH\), is an action of~\(\Grd\) on the arrow space of~\(\GrdH\) that commutes with the right translation action of~\(\GrdH\).  With this choice of category, \(\GrdS\)~is a functor from inverse semigroups to étale topological groupoids, which is left adjoint to~\(\Bisec\); that is, algebraic morphisms \(\GrdS(S)\Zak\Grd\) for a groupoid~\(\Grd\) correspond bijectively and naturally to inverse semigroup homomorphisms \(S\to \Bisec(\Grd)\).

The category of groupoids we use also has another important feature, which we discuss in Section~\ref{sec:morphisms_functors}.  Namely, an algebraic morphism \(\Grd\Zak\GrdH\) is equivalent to a functorial way to turn an \(\GrdH\)\nb-action on a topological space into a \(\Grd\)\nb-action on the same space.  Algebraic morphisms are the appropriate ones if we view groupoids as generalised groups, while functors are appropriate if we view groupoids as generalised spaces.  Groupoid \(\Cst\)\nb-algebras are functorial for algebraic morphisms, but not for continuous functors (see \cites{Buneci:Morphisms_dynamical, Buneci-Stachura:Morphisms_groupoids}); the orbit space is functorial for continuous functors, but not for algebraic morphisms.

\section{The terminal action of an inverse semigroup}
\label{sec:terminal}

The material included in this section follows ideas from lattice theory and related notions like the theory of locales and frames \cites{Johnstone:Stone_spaces, Pultr:Frames}.  However, since the idempotent semilattices of inverse semigroups are usually not lattices, we have chosen to give a self-contained exposition of the results we need.

We are going to describe a terminal object in the category of actions of an inverse semigroup~\(S\) on topological spaces.  This terminal action lives on a quasi-compact, sober, but non-Hausdorff topological space which is naturally associated to the semilattice \(E=E(S)\) of idempotents in the inverse semigroup.  As a preparation, we fix some details about inverse semigroups, their actions, and terminal objects.

Many inverse semigroups that arise in practice have a \emph{zero} and a \emph{unit}, that is, elements \(0,1\in S\) with \(0\cdot s=0=s\cdot 0\) and \(1\cdot s=s=s\cdot 1\) for all \(s\in S\).  If our inverse semigroup does not yet have them, we add a formal zero and a formal unit and extend the product by the rules above.  Thus it is no loss of generality to assume that inverse semigroups have unit and zero.  We always (and often tacitly) assume that homomorphisms preserve the unit and zero.

\begin{definition}
  A \emph{partial homeomorphism} of a topological space~\(X\) is a homeomorphism \(U\to V\) between open subsets \(U\) and~\(V\) of~\(X\).  These form an inverse semigroup with unit and zero with respect to the composition of partial maps.  The unit is the identity map on~\(X\), and the zero is the empty map.
\end{definition}

\begin{definition}
  \label{def:inverse_semigroup_act}
  An \emph{action} of an inverse semigroup~\(S\) on a topological space~\(X\) is a unit- and zero-preserving homomorphism from~\(S\) to the inverse semigroup of partial homeomorphisms of~\(X\).  We also write \(s\cdot x\) for the action of \(s\in S\) on an element~\(x\) in the domain of the partial homeomorphism associated to~\(s\).
\end{definition}

\begin{definition}
  Let \(X\) and~\(Y\) be topological spaces with actions of~\(S\).  A map \(f\colon X\to Y\) is called \emph{\(S\)\nb-equivariant} if, for \(s\in S\) and \(x\in X\), \(s\cdot x\) is defined if and only if \(s\cdot f(x)\) is defined, and then \(f(s\cdot x) = s\cdot f(x)\).
\end{definition}

Actions of~\(S\) on topological spaces with \(S\)\nb-equivariant continuous maps form a category, denoted~\(\Top^S\).

\begin{definition}
  \label{def:terminal}
  A \emph{terminal object} in a category is an object that admits a unique map from each object in the category.
\end{definition}

Such a terminal object is unique up to isomorphism if it exists.

\begin{example}
  The terminal object in an additive category is the zero object.
\end{example}

\begin{example}
  In the category of group actions, the terminal object is the trivial action on the one-point space.
\end{example}

\begin{example}
  The terminal object in the category of continuous actions of a topological groupoid~\(\Grd\) is more complicated: it is the object space~\(\Grd^{(0)}\) of~\(\Grd\), equipped with the canonical \(\Grd\)\nb-action (the anchor map for this action is the identity map on~\(\Grd^{(0)}\), and arrows act by \(g\cdot \s(g) = \r(g)\)).

  An action of~\(\Grd\) on a topological space~\(X\) is given by an anchor map \(\varrho\colon X\to\Grd^{(0)}\) and a multiplication map \(\mu\colon \Grd\times_{\s,\varrho} X\to X\) (see Definition~\ref{def:groupoid_act}).  The anchor map \(\varrho\colon X\to \Grd^{(0)}\) is equivariant by definition because \(\varrho(g\cdot x)=\r(g)=g\cdot \varrho(x)\), and it is routine to check that it is the only equivariant map \(X\to \Grd^{(0)}\).
\end{example}

At first sight, one may hope that terminal actions of inverse semigroups are trivial as in the group case.  It turns out, however, that this is not the case: inverse semigroups behave more like groupoids than groups in this respect.  A complicated topology on the terminal action is needed because the domains of idempotents in the inverse semigroup should be as independent as possible (compare Remark~\ref{rem:Ue_independent}).

\subsection{The character space of a semilattice}
\label{sec:character_sl}

The idempotent elements in an inverse semigroup~\(S\) form a semilattice \(E(S)\).  The underlying space of the terminal action of~\(S\) depends only on \(E(S)\), so that we work with semilattices for some time.  We assume throughout that they have a unit and a zero.

\begin{definition}
  \label{def:Open}
  For a topological space~\(X\), we let \(\Open(X)\) be the semilattice of open subsets of~\(X\) with multiplication~\(\cap\).
\end{definition}

\begin{definition}
  A \emph{character} on a semilattice~\(E\) is a homomorphism to \(\{0,1\}\).  Equivalently, it is the characteristic function of a (proper) filter, that is, a subset~\(F\) of~\(E\) that satisfies
  \begin{enumerate}
  \item \(0\notin F\), \(1\in F\);
  \item \(e,f\in F\Rightarrow ef\in F\); and
  \item \(e\in F\), \(f\in E\) with \(e\le f\) implies \(f\in F\).
  \end{enumerate}
  Let~\(\hat{E}\) be the set of characters on~\(E\).
\end{definition}

\begin{definition}
  \label{def:Ue}
  For \(e\in E\), we define \(U_e\defeq \bigl\{\varphi\in\hat{E} : \varphi(e)=1\bigr\} \subseteq \hat{E}\).
  We equip~\(\hat{E}\) with the topology generated by the subsets~\(U_e\).
\end{definition}

\begin{lemma}
  The map \(U\colon E\to \Open(\hat{E})\), \(e\mapsto U_e\), is an injective homomorphism of semilattices.  The subsets~\(U_e\) for \(e\in E\) form a basis for the topology on~\(\hat{E}\).
\end{lemma}

\begin{proof}
  Obviously, \(U_0=\emptyset\), \(U_1=\hat{E}\), and \(U_e\cap U_f = U_{e\cdot f}\), so that~\(U\) is a homomorphism.  It is injective because
  \begin{equation}
    \label{eq:special_character}
    \varphi_e(f) \defeq [e\le f] =
    \begin{cases}
      1&\text{if \(e\le f\),}\\0&\text{otherwise,}
    \end{cases}
  \end{equation}
  is a character that belongs to~\(U_f\) if and only if \(e \le f\).
\end{proof}

We frequently use the above notation~\([P]\) for a property~\(P\), which is~\(1\) if~\(P\) holds, and~\(0\) otherwise.

\begin{remark}
  \label{rem:Ue_independent}
  If \(U_e\subseteq U_{f_1}\cup \dotsb\cup U_{f_n}\) for idempotents
  \(e,f_1,\ldots,f_n\), then considering the character
  in~\eqref{eq:special_character} it follows that \(e\le f_j\) for
  some~\(j\).  This makes precise in which sense the subsets~\(U_e\)
  for \(e\in E\) are as independent as possible.
\end{remark}

\begin{lemma}
  \label{lem:open_hatE}
  Open subsets of~\(\hat{E}\) correspond bijectively to \emph{ideals} in~\(E\), that is, subsets \(I\subseteq E\) containing~\(0\) and satisfying \(ef\in I\) whenever \(e\in E\) and \(f\in I\) \textup(or equivalently, \(e\in I\) whenever \(f\in I\) and \(e\le f\)\textup).  If~\(A\) is an open subset of~\(\hat{E}\), then the associated ideal is \(I_A \defeq \{e\in E : U_e\subseteq A\}\).  And if~\(I\) is an ideal in~\(E\), the corresponding open subset of~\(\hat{E}\) is
  \[
  A_I \defeq  \bigcup_{e\in I} U_e
  = \{\varphi\in \hat{E} : \varphi(e)=1\text{ for some }e\in I\}.
  \]
\end{lemma}

The proof is straightforward.  The ideal corresponding to~\(U_e\) is the \emph{principal} ideal \(I_e \defeq \{f\in E : f\le e\}\) generated by~\(e\).

As a result, the set of ideals in a semilattice is isomorphic to \(\Open(\hat{E})\) and thus a locale (see \cite{Johnstone:Stone_spaces}*{II.1}).  The spectrum of this locale, consisting of its prime principal ideals, turns out to be homeomorphic to~\(\hat{E}\).  We find the above direct definition of~\(\hat{E}\) more convenient.

\begin{lemma}
  \label{lem:delta_map}
  For any topological space~\(X\), there is a natural continuous map
  \begin{equation}
    \label{eq:delta_map}
    \delta\colon X\to \widehat{\Open(X)},\qquad
    \delta(x)(U) \defeq [x\in U].
  \end{equation}
  It is a homeomorphism onto its image if~\(X\) is~\(T_0\).
\end{lemma}

\begin{proof}
  It is routine to check that~\(\delta(x)\) is a character on \(\Open(X)\) for \(x\in X\).  The map~\(\delta\) is continuous because \(\delta^{-1}(U_e) = e\) for all \(e\in \Open(X)\).  Thus \(\delta(e) = \delta(X)\cap U_e\), so that~\(\delta\) is open onto its image.  If~\(X\) is~\(T_0\), then~\(\delta\) is injective.
\end{proof}

Let \(\Semil\) be the category of semilattices (with unit and zero) and let \(\Topop\) be the opposite category of topological spaces.  The map \(E\mapsto\hat{E}\) is a covariant functor \(\Semil\to\Topop\): a semilattice homomorphism \(f\colon E\to E'\) induces a continuous map \(\widehat{E'} \to \widehat{E}\), \(\varphi\mapsto\varphi\circ f\).  In the converse direction, the map \(X\mapsto \Open(X)\) is part of a covariant functor \(\Topop\to\Semil\): a continuous map \(f\colon X'\to X\) induces a semilattice homomorphism \(\Open(f)\colon\Open(X)\to\Open(X')\), \(U\mapsto f^{-1}(U)\).

\begin{lemma}
  \label{lem:character_universal}
  For any topological space~\(X\), there is a natural bijection between continuous maps \(\hat{E}\leftarrow X\) and homomorphisms \(E\to \Open(X)\).  This bijection sends \(f\colon X\to\hat{E}\) to the map \(\check{f}\colon E\to\Open(X)\), \(e\mapsto f^{-1}(U_e)\), and it maps \(g\colon E\to\Open(X)\) to \(\hat{g}\colon X\to\hat{E}\) with \(\hat{g}(x)(e) \defeq [x\in g(e)]\).
\end{lemma}

\begin{proof}
  This is straightforward to verify by hand.  Instead of this argument, we reinterpret the assertion: we must check that
  \[
  \Topop(\hat{E},X) \cong \Semil(E,\Open(X))
  \]
  if~\(E\) is a semilattice and~\(X\) a topological space.  Here we write \(\Cat(x,y)\) for the set of arrows \(x\to y\) in a category~\(\Cat\).   This means that the functors \(\Open\) and~\(\hat{\ }\) are adjoint to each other.  The unit and counit of this adjunction are the natural transformations \(U^E\colon E\to \Open(\hat{E})\) and \(\delta^X\colon X\to \widehat{\Open(X)}\) described above.  The adjointness follows because the following composite maps are identities:
  \[
  \Open(X) \xrightarrow{U^{\Open(X)}}
  \Open\Bigl(\widehat{\Open(X)}\Bigr) \xrightarrow{\Open(\delta^X)}
  \Open(X),\qquad
  \hat{E} \xleftarrow{\widehat{U^E}}
  \Open\bigl(\hat{E}\bigr)\sphat\xleftarrow{\delta^{\hat{E}}}
  \hat{E}\qedhere
  \]
\end{proof}

\begin{definition}[see~\cites{Johnstone:Stone_spaces, Pultr:Frames}]
  \label{def:sober}
  A topological space is \emph{sober} if every irreducible closed non-empty subset is the closure of a unique point.  Here a closed subset~\(A\) is \emph{irreducible} if \(A=A_1\cup A_2\) with closed subsets \(A_i\subseteq \hat{E}\) implies \(A_1=A\) or \(A_2=A\).
\end{definition}

Sober spaces are~\(T_0\) but not necessarily Hausdorff.  Hausdorff spaces are sober.  There are spaces which are sober but not~\(T_1\), and \(T_1\)\nb-spaces which are not sober.

\begin{lemma}
  \label{lem:cs_sober}
  The space~\(\hat{E}\) is sober and locally quasi-compact.  Each subset~\(U_e\) for \(e\in E\) is quasi-compact.
\end{lemma}

\begin{proof}
  We claim that if \(\varphi,\psi\in \hat{E}\) satisfy \(\overline{\{\varphi\}}=\overline{\{\psi\}}\), then \(\varphi=\psi\), that is, \(\hat{E}\) is~\(T_0\).  Indeed, if \(\varphi\neq \psi\), then there is \(e\in E\) with \(\varphi(e)\neq \psi(e)\).  If, say, \(\varphi(e)=1\) and \(\psi(e)=0\), then \(\varphi\notin \overline{\{\psi\}}\) because~\(U_e\) is a neighbourhood of~\(\varphi\) not containing~\(\psi\).

  Let~\(A\) be a non-empty closed irreducible subset of~\(\hat{E}\).  Since~\(A\) is closed, we may write \(\hat{E}\setminus A=\bigcup_{e\in F} U_e\), where \(F=\{e\in E: U_e\cap A=\emptyset\}\).
  Let~\(\varphi\) be the characteristic function of \(E\setminus F\).  We claim that~\(\varphi\) is a character with \(A=\overline{\{\varphi\}}\).
  We have \(\varphi(0)=0\) because \(U_0=\emptyset\) and \(\varphi(1)=1\) because \(U_1=\hat{E}\) and~\(A\) is non-empty.
  If \(e\in F\), then \(ef\in F\) for all \(f\in E\) because \(U_{ef}=U_e\cap U_f\subseteq U_e\).
  In other words, if \(ef\notin F\), then \(e\notin F\) and \(f\notin F\).
  If \(e\notin F\) and \(f\notin F\), then \(A\cap U_e\neq \emptyset\) and \(A\cap U_f\neq \emptyset\).
  Since~\(A\) is irreducible, we have \(A\setminus U_e\cup A\setminus U_f=A\setminus U_{ef}\neq A\), so that \(A\cap U_{ef}\neq \emptyset\), that is, \(ef\notin F\).  Thus~\(\varphi\) is a character.  Given \(\psi\in \hat{E}\), we have \(\psi\notin \overline{\{\varphi\}}\) if and only if there is \(e\in E\) with \(\psi\in U_e\) and \(\varphi\notin U_e\), that is, \(e\in F\).  Thus \(\psi\notin \overline{\{\varphi\}}\) if and only if \(\psi\in \bigcup_{e\in F}U_e=\hat{E}\setminus A\), that is, \(\psi\notin A\).
  Hence \(A=\overline{\{\varphi\}}\).

  Finally, we check that~\(U_e\) is quasi-compact for each \(e\in E\).  This includes the special case \(\hat{E}=U_1\).  The case \(e=0\) is trivial, so that we may assume \(e\neq0\).  The character~\(\varphi_e\) defined in~\eqref{eq:special_character} belongs to~\(U_f\) if and only if \(e\le f\).  Thus no proper open subset of~\(U_e\) contains~\(\varphi_e\).  In an open covering of~\(U_e\), one subset contains~\(\varphi_e\) and is therefore equal to~\(U_e\).  This yields a subcovering with one element.
\end{proof}

The above proof shows that the only neighbourhood of \(\varphi_1\in\hat{E}\) is~\(\hat{E}\).
Thus~\(\hat{E}\) is non-Hausdorff unless \(E=\{0,1\}\) and~\(\hat{E}\) is the one-point space.

\begin{remark}
  \label{rem:sober}
  Mapping an open subset to its complement defines an isomorphism of semilattices from \(\Open(X)\) onto the semilattice of closed subsets with product~\(\cup\).  Since the set of irreducible closed subsets of~\(X\) is defined in terms of closed subsets and~\(\cup\) only, we may recover the underlying space~\(X\) from the semilattice \(\Open(X)\) provided~\(X\) is sober.  We cannot recover~\(X\) from \(\Open(X)\) unless~\(X\) is sober because for any space~\(X\) there is a continuous map \(X\to X'\) to a sober space~\(X'\) that induces an isomorphism \(\Open(X)\cong \Open(X')\) (see~\cite{Johnstone:Stone_spaces}*{II.1}).
\end{remark}

\begin{example}
  Let \(E=\{0,1\}\cup F\), where~\(F\) is either a finite set written as \(F=\{e_1,e_2,\dotsc,e_n\}\) or an infinite (countable) set of the form \(\{e_1,e_2,\dotsc\}\).  We consider two different types of semilattice structures on~\(E\).  In the semilattice~\(E_<\), we assume \(e_1<e_2<e_3<\dotsb\); in the semilattice~\(E_\bot\), we assume that the elements~\(e_i\) are orthogonal, that is, \(e_ie_j=0\) for all \(i\neq j\).

  In both cases, we consider the principal filter \(F_i \defeq \{e\in E : e\ge e_i\}\) generated by~\(e_i\) and let~\(\varphi_i\defeq 1_{F_i} \in\hat{E}\) be its characteristic function.  In~\(E_<\), we have \(F_i=\{e_i,e_{i+1},\dotsc,1\}\), whereas in~\(E_\bot\) we have \(F_i=\{e_i,1\}\).  In addition, \(F_\infty\defeq \{1\}\) is a filter and thus \(\varphi_\infty \defeq 1_{F_\infty}\) is a character.  These are all the characters (and filters) in both cases.  Thus \(\hat{E}=\{\varphi_1,\varphi_2,\dotsc,\varphi_\infty\}\) in both cases.  However, the topologies on \(\widehat{E_<}\) and~\(\widehat{E_\bot}\) are different.  We have
  \[
  \Open(\widehat{E_<})=\{\emptyset=U_0, \{\varphi_1\}=U_{e_1}, \{\varphi_1,\varphi_2\}=U_{e_2},\dotsc,U_{1}=\widehat{E_<}\}.
  \]
  In this case, the basis \(\{U_e : e\in E\}\) is already the whole topology of \(\widehat{E_<}\) and we have a semilattice isomorphism \(E_<\cong\Open(\widehat{E_<})\).  In~\(\widehat{E_\bot}\), however, we have \(U_{e_i}=\{\varphi_i\}\) for all \(i\), and we have to take unions in order to get the whole topology of~\(\widehat{E_\bot}\).  The map \(E_\bot\to \Open(\widehat{E_\bot})\) is therefore not surjective.  Topologise~\(F\) by the discrete topology, then \(\Open(\widehat{E_\bot})\) is isomorphic to \(\Open(F)\cup \{1_\infty\}\), that is, the lattice of all subsets of~\(F\) with an extra unit~\(1_\infty\) added (playing the role of~\(\widehat{E_\bot}\) in \(\Open(\widehat{E_\bot})\)).  Furthermore, \(\widehat{E_\bot}\)~is homeomorphic to the non-Hausdorff compactification \(F\cup\{\infty\}\) of~\(F\) where the only neighbourhood of~\(\infty\) is the whole space.
\end{example}

\subsection{The character space as a terminal action}
\label{sec:character_terminal}

Now we return to an inverse semigroup~\(S\) with unit and zero.
Let \(E=E(S)\) be its idempotent semilattice and \(\hat{E}\)~the character space defined above.  Recall that \(g^* eg \in S\) is idempotent for all \(e\in E\), \(g\in S\).  We define an action of~\(S\) on~\(\hat{E}\) by partial homeomorphisms: let \(g\in S\) act by
\[
c_g\colon U_{g^*g}\to U_{gg^*},\qquad
c_g(\varphi)(e) \defeq \varphi(g^* e g),
\]
with domains \(U_e\subseteq \hat{E}\) for \(e\in E\) as in Definition~\ref{def:Ue}.
The maps \(c_g\) and~\(c_{g^*}\) are both continuous, and they are inverse to each other because
\[
c_g \circ c_{g^*} (\varphi)(e) = \varphi(gg^* e gg^*) = \varphi(gg^*)^2\varphi(e) = \varphi(e)
\]
for \(\varphi\in U_{gg^*}\), \(e\in E\) and, similarly, \(c_{g^*} \circ c_g = \Id_{U_{g^*g}}\).  Furthermore, \(c_{gh} = c_g\circ c_h\), \(c_1=\Id_{\hat{E}}\), and \(c_0=\emptyset\), so that we have an inverse semigroup action.





\begin{theorem}
  \label{the:terminal_action}
  The action of~\(S\) on~\(\hat{E}\) described above is a terminal object in the category of actions of~\(S\) on topological spaces, that is, there is a unique \(S\)\nb-equivariant continuous map \(X\to\hat{E}\) for any topological space~\(X\) with an action of~\(S\).
\end{theorem}

\begin{proof}
  An action~\(\theta\) of~\(S\) on~\(X\) involves a semilattice map \(E\to\Open(X)\), mapping \(e\in E\) to the domain \(X_e\subseteq X\) of~\(\theta_e\).  By Lemma~\ref{lem:character_universal}, this induces a continuous map \(f\colon X\to\hat{E}\) given by \(f(x)(e)=[x\in X_e]\).  Notice that \(s\cdot x\) is defined if and only if \(x\in X_{s^*s}\) if and only if \(s\cdot f(x)\) is defined.  If \(x\in X_{s^*s}\), then \(f(s\cdot x)(e)=[s\cdot x\in X_e]\) and \(s\cdot f(x)(e)=[x\in X_{s^*es}]\), and these are equal because \(x\in X_{s^*es}\) if and only if \(s\cdot x\in X_e\) provided \(x\in X_{s^*s}\).  Hence \(f\colon X\to \hat{E}\) is \(S\)\nb-equivariant.

  Conversely, let \(g\colon X\to \hat{E}\) be any \(S\)\nb-equivariant, continuous map.  Given \(x\in X\) and \(e\in E\), we have \(g(x)(e)=1\) if and only if \(x\in X_e\) because \(x\in X_e\) if and only if \(e\cdot x\) is defined if and only if \(e\cdot g(x)\) is defined, if and only if \(g(x)\in U_e\) if and only if \(g(x)(e)=1\).  Thus \(g=f\).
\end{proof}

\begin{remark}
  The character space~\(\hat{E}\) with a more complicated (totally disconnected) \emph{Hausdorff} topology has already been considered by Paterson~\cite{Paterson:Groupoids}.  In his topology, the domains \(U_e\subseteq\hat{E}\) of the partial action on~\(\hat{E}\) are both open and closed.  The map \(X\to\hat{E}\) constructed above is continuous for Paterson's topology if and only if the domains \(X_e\subseteq X\) of the partial action on~\(X\) are both \emph{closed} and open.  This is a severe restriction, which rules out the inverse semigroup actions that appear in the study of foliations.
\end{remark}

We may also compare actions of \(E\) and~\(\hat{E}\) on \(\Cst\)\nb-algebras.  Let~\(A\) be a \(\Cst\)\nb-algebra.

An action of a semilattice~\(E\) with unit and zero on~\(A\) (in the sense of Definition~\ref{def:inverse_semigroup_act}) is a unit- and zero-preserving homomorphism from~\(E\) to the semilattice of ideals in~\(A\).  The ideal semilattice of~\(A\) is canonically isomorphic to the semilattice of open subsets of~\(\Prim A\), the primitive ideal space of~\(A\).  Thus an action of~\(E\) on~\(A\) is equivalent to an action of~\(E\) on \(\Prim A\).  (This is the point where semilattices are much easier than general inverse semigroups.)

By our previous results, an action of~\(E\) on \(\Prim A\) is equivalent to a continuous map \(\Prim A\to \hat{E}\).  This is exactly the structure that turns~\(A\) into a \(\Cst\)\nb-algebra over~\(\hat{E}\) in the sense of~\cite{Meyer-Nest:Bootstrap}.  And this defines how topological spaces act on \(\Cst\)\nb-algebras.  Summing up:

\begin{proposition}
  \label{pro:semilattice_act_Cstar}
  An action of a semilattice~\(E\) on a \(\Cst\)\nb-algebra is equivalent to an action of the topological space~\(\hat{E}\), that is, to a structure of\/ \(\Cst\)\nb-algebra over~\(\hat{E}\).
\end{proposition}

\section{The universal groupoid associated with an inverse semigroup}
\label{sec:universal_groupoid}

Given an action of an inverse semigroup~\(S\) on a topological space~\(X\), we get an associated étale topological groupoid of germs as in~\cites{Exel:Inverse_combinatorial,Paterson:Groupoids} (this differs from the construction of the germ groupoid in~\cite{Renault:Cartan.Subalgebras}).  We recall this construction for the action of~\(S\) on~\(\hat{E}\).  We denote the resulting groupoid by~\(\GrdS(S)\).

Its object space \(\GrdS(S)^{(0)}\) is~\(\hat{E}\).  Its arrows are equivalence classes of pairs \((s,\varphi)\) with \(s\in S\), \(\varphi\in U_{s^*s}\), where we identify \((s,\varphi)\) and \((t,\psi)\) if \(\varphi=\psi\) and there is \(e\in E(S)\) with \(s\cdot e= t\cdot e\) and \(\varphi\in U_e\) (that is, \(\varphi(e)=1\)).  We write \(\germ s\varphi\) for the equivalence class of \((s,\varphi)\).  The source and range maps of \(\GrdS(S)\) are defined by
\[
\s(\germ s\varphi)=\varphi\quad\text{and}\quad
\r(\germ s\varphi)=s\cdot \varphi,
\]
and the composition is \(\germ s\varphi\cdot \germ t\psi = \germ{s\cdot t}{\psi}\) if \(\varphi=t\cdot \psi\).  The inversion is \(\germ s\varphi^{-1}=\germ{s^*}{s\cdot\varphi}\), and the unit arrow at~\(\varphi\) is \(\germ 1\varphi\).

The topology on the arrow space~\(\GrdS(S)^{(1)}\) is the smallest one in which the subsets
\[
\OpenG(s,U)\defeq \{\germ s\varphi : \varphi\in U\}
\qquad\text{for \(s\in S\) and \(U\in \Open(U_{s^*s})\)}
\]
are open.  We abbreviate \(\OpenG_s \defeq \OpenG(s,U_{s^*s})\).

\begin{definition}
  \label{def:bisection}
  A \emph{bisection} of a topological groupoid~\(\Grd\) is an open subset of~\(\Grd^{(1)}\) on which the range and source maps are both injective and open.

  A groupoid is \emph{étale} if its arrow space is covered by bisections.
\end{definition}

We restrict to open bisections because we never use more general bijections.

\begin{lemma}
  \label{lem:OpenG_s_basis}
  Each~\(\OpenG_s\) is a bisection of\/~\(\GrdS(S)\), and these subsets for \(s\in S\) form a basis for the topology on~\(\GrdS(S)^{(1)}\). In particular, \/~\(\GrdS(S)\) is an étale groupoid.
\end{lemma}

\begin{proof}
  Since \([s,\varphi]=[t,\psi]\) implies \(\varphi=\psi\) and \(s\cdot\varphi=t\cdot\psi\), the source and range maps restrict to bijections on~\(\OpenG_s\), which map \(\OpenG(s,U)\) onto \(U\) and~\(s\cdot U\), respectively.  Hence they are open and injective on~\(\OpenG_s\), so that~\(\OpenG_s\) is a bisection.  Since these bisections cover~\(\GrdS(S)^{(1)}\), \(\GrdS(S)\) is étale.

  We claim that if \(\OpenG(s,U)\cap \OpenG(t,V)\) contains~\(\germ s\varphi\), then it contains a neighbourhood of~\(\germ s\varphi\) of the form~\(\OpenG_u\) for some \(u\in S\).  First there is \(e\in E\) with \(se=te\) and \(\varphi\in U_e\) because \(\germ s\varphi=\germ t\varphi\).  Let \(W \defeq U\cap V\cap U_e\), this is a neighbourhood of~\(\varphi\) in~\(\hat{E}\).  It contains~\(U_f\) for some~\(f\in E\) with \(f\le e\) because \(U_f\subseteq U_e\).  Then \(sf=tf\), and \(\OpenG_{sf}\subseteq \OpenG(s,U)\cap \OpenG(t,V)\).  It follows that if \(\germ s\varphi\in \bigcap_{i=1}^n \OpenG(s_i,U_i)\), then there is \(u\in S\) with \(\germ s\varphi\in \OpenG_u \subseteq \bigcap_{i=1}^n \OpenG(s_i,U_i)\).  Thus the subsets~\(\OpenG_u\) form a basis for the topology on \(\GrdS(S)\).
\end{proof}

\begin{lemma}
  \label{lem:GrdS_sober_compact}
  The arrow space of \(\GrdS(S)\) is locally quasi-compact and sober.
\end{lemma}

\begin{proof}
  In general, the arrow space of an étale groupoid~\(\Grd\) is sober or locally quasi-compact if and only if its unit space~\(\Grd^{(0)}\) is sober or locally quasi-compact.

  This is clear for local quasi-compactness.  We argue for sobriety.  On the one hand, \(\Grd^{(0)}\)~is open in~\(\Grd^{(0)}\), and open (or more generally, locally closed) subspaces of sober spaces are again sober.  On the other hand, if \(f\colon X\to Y\) is a local homeomorphism of topological spaces, then~\(X\) is sober provided~\(Y\) is.

  Finally, since~\(\hat{E}\), the object space of~\(\GrdS(S)\), is locally quasi-compact and sober by Lemma~\ref{lem:cs_sober}, so is its arrow space.
\end{proof}

\begin{definition}
  \label{def:Bisec}
  Given an étale topological groupoid~\(\Grd\), let \(\Bisec(\Grd)\) be the \emph{inverse semigroup of bisections} of~\(\Grd\), with multiplication \(s\cdot t = \{g\cdot h : g\in s,\ h\in t\}\) and pseudoinverse \(s^*=\{g^{-1} : g\in S\}\).
\end{definition}

\begin{lemma}
  \label{lem:S_inverse_semigroup_of_bisections_of_G(S)}
  The map \(S\to \Bisec\big(\GrdS(S)\big)\), \(s\mapsto \OpenG_s\), is an injective homomorphism.
\end{lemma}

\begin{proof}
  It is easy to see that \(\OpenG_s\cdot\OpenG_t=\OpenG_{st}\), \(\OpenG_0=\emptyset\) and \(\OpenG_1=\GrdS(S)^{(0)} = \hat{E}\), so that the map \(s\mapsto \OpenG_s\) is a homomorphism of inverse semigroups with unit and zero.

  To prove injectivity, assume \(\OpenG_s=\OpenG_t\).  Then \(U_{s^*s}=U_{t^*t}\) and hence \(s^*s=t^*t\).  Let \(e\defeq s^*s=t^*t\) and define \(\varphi_e(f) \defeq [f\ge e]\) as in Equation~\eqref{eq:special_character}.  Then \(\varphi_e\in U_e\), so that \(\germ s{\varphi_e}=\germ t{\varphi_e}\) because \(\OpenG_s=\OpenG_t\).  Therefore, there is \(f\in E\) with \(\varphi\in U_f\) (that is, \(f\ge e\)) and \(sf=tf\).  Now \(f\ge e\) implies \(s=sf=tf=t\).
\end{proof}

\begin{definition}
  \label{def:groupoid_act}
  Let~\(\Grd\) be a (possibly non-Hausdorff) topological groupoid, with arrow space~\(\Grd^{(1)}\), object space~\(\Grd^{(0)}\) and range and source maps \(\r,\s\colon \Grd^{(1)}\rightrightarrows \Grd^{(0)}\).

  An \emph{action} of~\(\Grd\) on a (possibly non-Hausdorff) topological space~\(X\) is a pair \((\varrho,\mu)\), where~\(\varrho\) is a continuous map \(X\to \Grd^{(0)}\), called \emph{anchor map}, and~\(\mu\) is a continuous map \(\mu\colon \Grd^{(1)}\times_{\s,\varrho} X\defeq\{(g,x)\in \Grd^{(1)}\times X: \s(g)=\varrho(x)\} \to X\), written \((g,x)\mapsto g\cdot x\), such that \(\varrho(g\cdot x) = \r(g)\) for all \(g\in \Grd^{(1)}\), \(x\in X\) with \(\s(g)=\varrho(x)\), \(g_1\cdot (g_2\cdot x) = (g_1\cdot g_2)\cdot x\) if both sides are defined, and \(1_{\varrho(x)}\cdot x=x\) for all \(x\in X\).

  A map \(f\colon X\to Y\) between two spaces \(X\) and~\(Y\) with such actions of~\(\Grd\) is \emph{\(\Grd\)\nb-equivariant} if \(\varrho_Y\circ f=\varrho_X\) and \(f(g\cdot x) = g\cdot f(x)\) for all \(g\in \Grd^{(1)}\), \(x\in X\) with \(\s(g)=\varrho_X(x)\).  That is, \(f(g\cdot x)\) is defined if and only if \(g\cdot f(x)\) is defined, and both are equal if defined, for all \(g\in \Grd^{(1)}\), \(x\in X\).

  With \(\Grd\)\nb-equivariant maps, the actions of~\(\Grd\) form a category, denoted \(\Top^\Grd\).
\end{definition}

\begin{theorem}
  \label{the:act_spaces_equivalent}
  The categories \(\Top^S\) and \(\Top^{\GrdS(S)}\) of actions of~\(S\) and \(\GrdS(S)\) on topological spaces are isomorphic.
\end{theorem}

\begin{proof}
  Let~\(X\) be a space with an action of~\(S\).  By Lemma~\ref{lem:character_universal}, we get a continuous \(S\)\nb-equivariant map \(\varrho\colon X\to\hat{E} = \GrdS(S)^{(0)}\) given by \(\varrho(x)(e)=[x\in X_e]\), which we take as our anchor map.  If \(\varrho(x)=\varphi\) and \(g\in\GrdS(S)^{(1)}\) satisfies \(\s(g)=\varphi\), then \(g=\germ s\varphi\) for some \(s\in S\), so that \(\varphi(s^*s)=1\) and hence \(x\in X_{s^*s}\), that is, \(s\cdot x\) is defined.  To define \(\germ s\varphi\cdot x\defeq s\cdot x\), we must check that this does not depend on the choice of the representative~\((s,\varphi)\).  If \(\germ t\varphi= \germ s\varphi\), then there is \(e\in E\) such that \(\varphi(e)=1\) and \(se=te\).  Since \(e\cdot \varphi\) is defined, so is \(e\cdot x\), and \(e\cdot x=x\).  Thus \(s\cdot x=s\cdot (e\cdot x)=(se)\cdot x=(te)\cdot x=t\cdot (e\cdot x)=t\cdot x\).  Thus an action of~\(S\) on~\(X\) yields an action of~\(\GrdS(S)\) on~\(X\).

  Conversely, an action of the étale groupoid~\(\GrdS(S)\) on a topological space~\(X\) induces an action of the inverse semigroup \(\Bisec\big(\GrdS(S)\big)\).  Using the homomorphism \(s\mapsto \OpenG_s\) from Lemma~\ref{lem:S_inverse_semigroup_of_bisections_of_G(S)}, we may turn this into an action of~\(S\).  More precisely, given \(e\in E\), let \(X_e\defeq \varrho^{-1}(U_e)\) and for each \(s\in S\), define \(\theta_s\colon X_{s^*s}\to X_{ss^*}\) by \(\theta_s(x)\defeq g\cdot x\), where~\(g\) is the unique element of the bisection~\(\OpenG_s\) with \(\s(g)=\varrho(x)\).  This defines an inverse semigroup action of~\(S\) on~\(X\).

  It is easy to check that these constructions provide functors \(\Top^S\leftrightarrows\Top^{\GrdS(S)}\) that are inverse to each other.
\end{proof}

Could there be a Hausdorff topological groupoid whose actions on Hausdorff topological spaces are equivalent to actions of~\(S\)?  Unfortunately, the answer is no, already for the simplest conceivable inverse semigroup, the semilattice \(\{0,e,1\}\) with \(e^2=e\).  An action of~\(E\) on a topological space~\(X\) is the same as specifying an open subset \(U_e\subseteq X\), the domain of~\(e\).  Thus an action of~\(E\) on a Hausdorff space is a Hausdorff space with an open subset.  Taking its characteristic function, an open subset is equivalent to a continuous map from~\(X\) to \(\{0,1\}\) with the topology where \(\{1\}\)~is open and \(\{0\}\)~is not open (this so-called Sierpi\'nski space is homeomorphic to~\(\hat{E}\)).  There is, however, no Hausdorff space~\(Y\) such that open subsets in Hausdorff spaces~\(X\) correspond to maps \(X\to Y\).

Using non-Hausdorff spaces becomes more natural when we study actions on \(\Cst\)\nb-algebras because their primitive ideal spaces are usually non-Hausdorff.  In Proposition~\ref{pro:semilattice_act_Cstar} we already observed that actions of a semilattice~\(E\) on \(\Cst\)\nb-algebras are equivalent to actions of the topological space~\(\hat{E}\).  For a general inverse semigroup, we therefore expect actions of~\(S\) on \(\Cst\)\nb-algebras to be equivalent to continuous actions of~\(\GrdS(S)\).  This is indeed the case for an appropriate definition of continuous action for non-Hausdorff groupoids.  We plan to discuss this definition elsewhere.

\section{Functoriality}
\label{sec:functoriality}

We want to establish that \(S\mapsto \GrdS(S)\) and \(\Grd\mapsto \Bisec \Grd\) is an adjoint pair of functors.  For this, we first have to describe the categories of inverse semigroups and groupoids we use.  One of them is fairly obvious:

\begin{definition}
  \label{def:invse}
  Let \(\Invse\) be the category of inverse semigroups with zero and unit, with homomorphisms (preserving zero and unit) as arrows.
\end{definition}

\subsection{The category of étale topological groupoids and functoriality of bisections}
\label{sec:Grpds}

The most obvious choice for a category of étale topological groupoids uses continuous functors as morphisms.  However, for this category neither \(\GrdS\) nor~\(\Bisec\) are functorial.  The correct arrows for the category of groupoids are the following:

\begin{definition}[\cites{Buneci:Groupoid_categories, Buneci-Stachura:Morphisms_groupoids}]
  \label{def:groupoid_category}
  Let \(\Grd\) and~\(\GrdH\) be topological groupoids.  An \emph{algebraic morphism} from~\(\Grd\) to~\(\GrdH\), denoted \(\Grd\Zak \GrdH\), is an action of~\(\Grd\) on~\(\GrdH^{(1)}\) that commutes with the action of~\(\GrdH\) on~\(\GrdH^{(1)}\) by right translations.
\end{definition}

Buneci and Stachura trace this definition back to Zakrzewski~\cite{Zakrzewski:Quantum_classical_I}.  They study algebraic morphisms because groupoid \(\Cst\)\nb-algebras are functorial for algebraic morphisms, but not for functors.

The following more concrete description of algebraic morphisms is already proved in \cite{Buneci:Groupoid_categories}*{Lemmas 2.8 and~2.9}.

\begin{lemma}
  \label{lem:algebraic_morphism_as_functor}
  An algebraic morphism \(\Grd\Zak\GrdH\) is equivalent to a pair consisting of an action of~\(\Grd\) on the object space~\(\GrdH^{(0)}\) of~\(\GrdH\) and a functor from the transformation groupoid \(\Grd\ltimes \GrdH^{(0)}\) to~\(\GrdH\) that acts identically on objects.
\end{lemma}

Recall that the groupoid \(\Grd\ltimes \GrdH^{(0)}\) has object space~\(\GrdH^{(0)}\), arrow space
\[
\Grd^{(1)}\times_{\s,\varrho} \GrdH^{(0)}=\{(g,x)\in \Grd^{(1)}\times \GrdH^{(0)} : \s(g)=\varrho(x)\},
\]
and composition \((g,h\cdot x)\cdot (h,x) \defeq (g\cdot h,x)\).

\begin{proof}
  First we explain how an algebraic morphism \(\Grd\Zak\GrdH\) induces a continuous action of~\(\Grd\) on the object space~\(\GrdH^{(0)}\).  We have \(\varrho(h)=\varrho(hh^{-1})=\varrho(1_{\r(h)})\), so that the anchor map \(\GrdH^{(1)}\to\Grd^{(0)}\) of the algebraic morphism is the composition of the range map \(\r\colon \GrdH^{(1)}\to\GrdH^{(0)}\) and a continuous map \(\varrho\colon \GrdH^{(0)}\to\Grd^{(0)}\), the restriction of the anchor map to units.  The action of~\(\Grd\) on~\(\GrdH^{(0)}\) has anchor map~\(\varrho\) and is defined uniquely by \(g\cdot \r(h)\defeq \r(g\cdot h)\) for all \(g\in\Grd^{(1)}\), \(h\in\GrdH^{(1)}\).  This action is continuous because \(g\cdot x = \r(g\cdot 1_x)\) for all \(g\in \Grd\) and \(x\in \GrdH^{(0)}\).  Since the action of~\(\Grd\) on~\(\GrdH^{(1)}\) commutes with right translations, given \(g\in\Grd^{(1)}\) and \(x\in\GrdH^{(0)}\) with \(\varrho(x)=\s(g)\), there is a unique \(\mu(g,x)\in\GrdH^{(1)}\) such that \(\s(\mu(g,x)) = x\) and \(g\cdot h = \mu(g,x)\cdot h\) for all \(h\in\GrdH^{(1)}\) with \(\r(h)=x\), namely, \(\mu(g,x) \defeq g\cdot 1_x\).  The map~\(\mu\) defines a continuous functor \(\Grd\ltimes \GrdH^{(0)} \to \GrdH\).  The above reasoning may be reversed, showing that \(\varrho\) and~\(\mu\) as above provide an algebraic morphism \(\Grd\Zak\GrdH\) by \(g\cdot h\defeq \mu\bigl(g,\r(h)\bigr)\cdot h\).
\end{proof}

\begin{example}
  \label{exa:alg_morphism_space}
  If \(\Grd\) and~\(\GrdH\) are just spaces (all arrows are identities), then an algebraic morphism \(\Grd\Zak\GrdH\) is the same as a continuous map \(\GrdH^{(0)}\to\Grd^{(0)}\).  More generally, if~\(\GrdH\) is a space and~\(\Grd\) arbitrary, then an algebraic morphism \(\Grd\Zak\GrdH\) is the same as a continuous action of~\(\Grd\) on the space \(\GrdH^{(0)}=\GrdH^{(1)}\).
\end{example}

\begin{example}
  \label{exa:alg_morphism_group}
  If \(\Grd\) and~\(\GrdH\) are both topological groups, then an algebraic morphism \(\Grd\Zak\GrdH\) is the same as a continuous group homomorphism \(\Grd\to\GrdH\).
\end{example}

\begin{example}
    The canonical left action of~\(\Grd\) on~\(\Grd^{(1)}\) is a morphism \(\Grd\Zak\Grd\), the \emph{identity} morphism.
\end{example}

\begin{remark}
  Algebraic morphism must be distinguished from Hilsum--Skandalis morphisms, also called Morita morphisms.  These are given by a topological space~\(X\) with a left \(\Grd\)- and a right \(\GrdH\)\nb-action, such that
  \begin{enumerate}
  \item the actions of \(\Grd\) and~\(\GrdH\) commute;
  \item the right \(\GrdH\)\nb-action is free and proper;
  \item the anchor map \(X\to\Grd^{(0)}\) for the left action induces a homeomorphism \(X/\GrdH \cong \Grd^{(0)}\).
  \end{enumerate}
  Given an algebraic morphism, we may let~\(\GrdH\) act on \(X\defeq\GrdH^{(1)}\) by right translations and~\(\Grd\) as specified on the left.  These two actions commute by assumption; the right \(\GrdH\)\nb-action on~\(X\) is free and proper with \(X/\GrdH \cong \GrdH^{(0)}\) via the range map.  The induced map \(X/\GrdH \to \Grd^{(0)}\) is the anchor map of the action of~\(\Grd\) on~\(\GrdH^{(0)}\) in Lemma~\ref{lem:algebraic_morphism_as_functor}.  Thus an algebraic morphism is a Hilsum--Skandalis morphism as well if and only if the induced map \(\GrdH^{(0)}\to\Grd^{(0)}\) is a homeomorphism.  Equivalently, it comes from a continuous functor that acts by a homeomorphism on objects.
\end{remark}

Recall that \(Y\times_\Grd X\) for a right \(\Grd\)\nb-space~\(Y\) with anchor map~\(\varrho_Y\) and a left \(\Grd\)\nb-space~\(X\) with anchor map~\(\varrho_X\) is the quotient of \(Y\times_{\varrho_X,\varrho_Y} X\) by the equivalence relation
\[
(y\cdot g,x)\sim (y,g\cdot x)
\]
for all \(y\in Y\), \(g\in\Grd^{(1)}\), \(x\in X\) with \(\varrho_Y(y) = \s(g)\) and \(\varrho_X(x) = \r(g)\).  The action map \((g,x)\mapsto g\cdot x\) provides a natural homeomorphism \(\Grd\times_\Grd X\cong X\) for any left \(\Grd\)\nb-space~\(X\), where we let~\(\Grd\) act on itself by right translations.

This allows us to compose algebraic morphisms: given left \(\Grd_j\)\nb-actions on~\(\Grd_{j+1}\) commuting with the right \(\Grd_{j+1}\)\nb-action by translations for \(j=1,2\), we may use the homeomorphism \(\Grd_2 \times_{\Grd_2} \Grd_3 \cong \Grd_3\) to transform the induced left \(\Grd_1\)\nb-action on \(\Grd_2 \times_{\Grd_2} \Grd_3\) into one on~\(\Grd_3\), which commutes with the right translation action and hence defines an algebraic morphism \(\Grd_1\Zak\Grd_2\).

\begin{definition}
  \label{def:Grpds_category}
  Let \(\Grpds\) be the category whose objects are the étale topological groupoids and whose morphisms are the algebraic morphisms, with the composition and identities just described.
\end{definition}

Our next goal is to explain how an algebraic morphism \(\Grd\Zak\GrdH\) induces a homomorphism \(\Bisec \Grd \to \Bisec \GrdH\), so that we get a functor \(\Bisec\colon \Grpds\to\Invse\).  We first describe this functoriality of~\(\Bisec\) by hand and then more conceptually.  Notice that a continuous functor \(\Grd\to\GrdH\) does not induce a map on the level of bisections.

Let \(f=(\varrho,\mu)\) describe an algebraic morphism \(\Grd\Zak\GrdH\), where \(\varrho\colon\GrdH^{(1)}\to\Grd^{(0)}\) and~\(\mu\) is the functor \(\Grd\ltimes\GrdH^{(0)}\to\GrdH\) as in Lemma~\ref{lem:algebraic_morphism_as_functor}.  For a bisection \(t\subseteq\Grd^{(1)}\), let
\[
f_*(t) \defeq \{\mu(g,x) : g\in t,\ x\in \GrdH^{(0)},\ \s(g) = \varrho(x)\}
\]
Since \(\s\big(\mu(g,x)\big)= x\) and \(g\in t\) with \(\s(g) = \varrho(x)\) is unique if it exists, the source map is a bijection from \(f_*(t)\) onto \(\varrho^{-1}(\s(t))\).
Since \(f_*(t^{-1}) = f_*(t)^{-1}\), the range map is a bijection from \(f_*(t)\) onto \(\varrho^{-1}(\r(t))\).  Thus \(f_*(t)\) is a bisection in~\(\GrdH\).  We leave it to the reader to check that \(f_*\colon \Bisec \Grd\to\Bisec \GrdH\) is a homomorphism and preserves zero and unit, and that \(f_*\circ g_* = (f\circ g)_*\) for composable algebraic morphisms and \(\Id_*=\Id\).

\begin{definition}
  \label{def:equivariant_partial_homeo}
  Let~\(\Grd\) be an étale topological groupoid and let~\(X\) be a right \(\Grd\)\nb-space.  A partial homeomorphism \(t\colon U\to V\) of~\(X\) is (right) \emph{equivariant} if \(U\) is \(\Grd\)\nb-invariant and \(t(x\cdot g) = t(x)\cdot g\) for all \(x\in U\), \(g\in\Grd^{(1)}\) for which \(x\cdot g\) is defined.
\end{definition}

Equivariant partial homeomorphisms are closed under composition and inversion of partial homeomorphisms, so that they form an inverse semigroup.

\begin{lemma}
  \label{lem:equivariant_partial_homeo}
  The inverse semigroup of equivariant partial homeomorphisms of~\(\Grd^{(1)}\) with right translation action of~\(\Grd\) is isomorphic to \(\Bisec \Grd\).
\end{lemma}

\begin{proof}
  Given a bisection \(T\subseteq\Grd^{(1)}\), we define a partial map~\(t\) on~\(\Grd^{(1)}\) by \(t(h) \defeq g\cdot h\) if there is \(g\in T\) with \(\s(g)=\r(h)\) (\(g\) is unique if it exists).  This defines an equivariant partial homeomorphism of~\(\Grd^{(1)}\). Conversely, if \(t\colon U\to V\) is an equivariant partial homeomorphism of~\(\Grd^{(1)}\), then
  \[
  T\defeq \{g\in\Grd^{(1)}: \text{there is \(h\in U\) with \(g=t(h)h^{-1}=t(\r(h))\)}\}=t(\r(U))
  \]
  is a bisection with \(t(h)=g\cdot h\) whenever \(h\in U\), \(g\in T\) and \(\s(g)=\r(h)\).
\end{proof}

Let~\(X\) be a right \(\GrdH\)\nb-space and let~\(M\) be an \(\GrdH,\Grd\)-bispace.  Then \(X\times_\GrdH M\) is a right \(\Grd\)\nb-space.  An \(\GrdH\)\nb-equivariant partial homeomorphism \(t\colon U\to V\) of~\(X\) induces a \(\Grd\)\nb-equivariant partial homeomorphism of~\(X\times_\GrdH M\) by \(t_*(x,m) \defeq (t(x),m)\), with domain \(U\times_\GrdH M\subseteq X\times_\GrdH M\).  The map \(t\mapsto t_*\) is an inverse semigroup homomorphism preserving zero and unit.  This construction is functorial if we view bispaces~\(M\) as morphisms and compose them by the balanced product, \(M\circ N\defeq M\times_\Grd N\).

In particular, if \(M=\GrdH\) with the right translation action of~\(\GrdH\) and the left action given by an algebraic morphism \(\Grd\Zak\GrdH\), then a \(\Grd\)\nb-equivariant bisection of~\(\Grd\) induces an \(\GrdH\)\nb-equivariant bisection of \(\Grd\times_\Grd \GrdH \cong \GrdH\).  By Lemma~\ref{lem:equivariant_partial_homeo}, this yields a homomorphism \(\Bisec \Grd\to\Bisec\GrdH\).  This is the conceptual explanation of the functoriality of bisections promised above.  It is straightforward to see that this abstract construction is equivalent to the concrete one above.

\subsection{An adjointness relation}
\label{sec:functorial_GrdS}

\begin{theorem}
  \label{the:arrow_out_of_GrdS}
  Let~\(S\) be an inverse semigroup with zero and unit and let~\(\Grd\) be an étale topological groupoid.  Then algebraic morphisms \(\GrdS(S)\Zak \Grd\) correspond bijectively to homomorphisms \(S\to\Bisec \Grd\).
\end{theorem}

\begin{proof}
  Theorem~\ref{the:act_spaces_equivalent} has an equivariant analogue with the same proof: actions of~\(S\) on a right \(\Grd\)\nb-space by equivariant partial homeomorphisms are equivalent to actions of~\(\GrdS(S)\) that commute with the right \(\Grd\)\nb-action.  In particular, an algebraic morphism \(\GrdS(S)\Zak \Grd\) is equivalent to an action of~\(S\) on~\(\Grd^{(1)}\) by \(\Grd\)\nb-equivariant partial homeomorphisms.  By Lemma~\ref{lem:equivariant_partial_homeo}, this is equivalent to a homomorphism \(S\to\Bisec \Grd\).
\end{proof}

Theorem~\ref{the:arrow_out_of_GrdS} asserts that~\(\GrdS\) is left adjoint to the functor~\(\Bisec\).  In particular, it implies that~\(\GrdS\) is a functor \(\Invse\to\Grpds\).  Let us see this more directly.

An inverse semigroup homomorphism \(f\colon S\to T\) restricts to a homomorphism between the idempotent semilattices, and hence induces a continuous map \(\varrho\colon \widehat{E(T)}\to\widehat{E(S)}\) by Lemma~\ref{lem:character_universal}.  We let~\(S\) act on \(\widehat{E(T)}\) by composing~\(f\) with the canonical action of~\(T\) on \(\widehat{E(T)}\).  The map \(\germ s\varphi \mapsto \germ {f(s)}\varphi\) is a well-defined functor \(\GrdS(S)\ltimes \widehat{E(T)} \to \GrdS(T)\).  This defines an algebraic morphism \(\GrdS(S)\Zak\GrdS(T)\).  The resulting left action is defined by
\[
\germ s{\varrho(\varphi)} \cdot \germ t\varphi \defeq \germ{f(s)\cdot t}{\varphi}
\qquad
\text{for \(s\in S\), \(t\in T\), \(\varphi\in U_{t^*t}\subseteq \widehat{E(T)}\), \(\varrho(\varphi)\in U_{s^*s}\).}
\]

We also describe the adjointness between \(\Bisec\) and~\(\GrdS\) using a unit and counit of adjunction.  These are canonical arrows \(S\to \Bisec\big(\GrdS(S)\big)\) and \(\GrdS \Bisec(\Grd) \Zak \Grd\).

The first is the canonical inverse semigroup homomorphism \(\OpenG\colon S\to \Bisec\big(\GrdS(S)\big)\) in Lemma~\ref{lem:S_inverse_semigroup_of_bisections_of_G(S)}.

For the latter, we first need a map \(\Grd^{(0)} \to (\GrdS \Bisec \Grd)^{(0)}\).  The idempotent bisections of~\(\Grd\) are exactly the open subsets of the object space~\(\Grd^{(0)}\).  Hence the object space of \(\GrdS \Bisec (\Grd)\) is
\[
(\GrdS \Bisec \Grd)^{(0)} = \widehat{E(\Bisec \Grd)} = \widehat{\Open(\Grd^{(0)})}.
\]
The map \(\delta\colon \Grd^{(0)} \to \widehat{\Open(\Grd^{(0)})}\) that we need is defined already in~\eqref{eq:delta_map}.  Bisections of~\(\Grd\) act on~\(\Grd^{(0)}\) by partial homeomorphisms.  This defines an action of~\(\GrdS \Bisec(\Grd)\) on~\(\Grd^{(0)}\) by Theorem~\ref{the:act_spaces_equivalent}.  The map~\(\delta\) is clearly \(\Bisec \Grd\)\nb-equivariant, hence it is \(\GrdS \Bisec(\Grd)\)-equivariant as well.

Now we assume that~\(\Grd^{(0)}\) is~\(T_0\) to simplify the construction.  Then~\(\delta\) is a homeomorphism onto its image.  The topological groupoid \(\GrdS \Bisec(\Grd)\ltimes \Grd^{(0)}\) is isomorphic to the restriction of \(\GrdS \Bisec(\Grd)\) to~\(\Grd^{(0)}\).  This restriction is exactly the original groupoid~\(\Grd\).  This provides an isomorphism \(\GrdS \Bisec(\Grd)\ltimes \Grd^{(0)} \cong \Grd\) and hence an algebraic morphism \(\GrdS \Bisec(\Grd)\Zak \Grd\).  This is the counit of the adjunction in Theorem~\ref{the:arrow_out_of_GrdS}.

\subsection{Algebraic morphisms as functors of action categories.}
\label{sec:morphisms_functors}

If we think of groupoids as generalised spaces, then we would expect a morphism between them to induce a map between the spaces of orbit (isomorphism classes of objects) because this is the classical space described by a groupoid.  Whereas functors between groupoids do this, algebraic morphisms, in general, do not induce a map between the orbit spaces.

If we think of groupoids as symmetries of spaces, then we would expect a morphism to turn an action of one groupoid into an action of the other, in the same way as for group homomorphisms.  Whereas functors do not do this, we are going to show, even more, that algebraic morphisms are exactly the same as functors between the categories of actions that do not change the underlying space.

Let \(\Forget\colon \Top^\Grd\to\Top\) be the functor that forgets the \(\Grd\)\nb-action.

\begin{theorem}
  \label{the:algebraic_morphism_as_functor}
  An algebraic morphism \(\Grd\Zak\GrdH\) induces a functor \(F\colon \Top^\GrdH\to\Top^\Grd\) that satisfies \(\Forget\circ F=\Forget\).  Conversely, any functor with this property comes from an algebraic morphism \(\Grd\Zak\GrdH\) in this way.
\end{theorem}

\begin{proof}
  An algebraic morphism \(\Grd\Zak\GrdH\) induces a left \(\Grd\)\nb-action on \(\GrdH^{(1)}\times_\GrdH X\) for any \(\GrdH\)\nb-space~\(X\).  The latter is naturally homeomorphic to~\(X\) via the map \((h,x)\mapsto h\cdot x\).  Thus an \(\GrdH\)\nb-action becomes a \(\Grd\)\nb-action on the same space.  Since \(\GrdH\)\nb-equivariant maps are also \(\Grd\)\nb-equivariant, we get a functor \(F\colon \Top^\GrdH\to\Top^\Grd\) with \(\Forget\circ F=\Forget\).

  Conversely, take such a functor~\(F\).  When we apply~\(F\) to the space~\(\GrdH^{(1)}\) with left translation action of~\(\GrdH\), we get a left \(\Grd\)\nb-action on~\(\GrdH^{(1)}\).  We claim that this left action commutes with the right translation action, so that we get an algebraic morphism \(\Grd\Zak\GrdH\), and that the functor~\(F\) acts on any space in the way described above, given by this algebraic morphism.

  For \(x\in\GrdH^{(0)}\), let \(\GrdH_x\defeq \{g\in\GrdH^{(1)}: \s(g)=x\}\).  This is an \(\GrdH\)\nb-invariant subspace for the left translation action.  Since~\(F\) is a functor, \(\GrdH_x\)~is \(\Grd\)\nb-invariant as well, and the induced \(\Grd\)\nb-action on~\(\GrdH_x\) is the restriction of the \(\Grd\)\nb-action on~\(\GrdH^{(1)}\).  An arrow \(h\in\GrdH^{(1)}\) induces an \(\GrdH\)\nb-equivariant map \(\GrdH_{\r(h)} \to \GrdH_{\s(h)}\) by right translation.  Since~\(F\) is a functor, these maps remain \(\Grd\)\nb-equivariant, that is, the left \(\Grd\)\nb-action commutes with the right translation action of~\(\GrdH\).

  For an \(\GrdH\)\nb-space~\(X\) with anchor map~\(\pi\), the action is an \(\GrdH\)\nb-equivariant map \(\GrdH^{(1)} \times_{\s,\pi} X\to X\) if we let~\(\GrdH\) act on~\(\GrdH^{(1)}\) by left translations as above.  Since this map is surjective, the left \(\Grd\)\nb-action on~\(X\) is determined by the action on \(\GrdH^{(1)} \times_{\s,\pi} X\).  For each \(x\in X\), we get an \(\GrdH\)\nb-invariant subspace in \(\GrdH^{(1)} \times_{\s,\pi} X\) consisting of elements of the form \((h,x)\), which is isomorphic to~\(\GrdH^{(1)}_{\pi(x)}\).  Hence the action of~\(\Grd\) on \(\GrdH^{(1)} \times_{\s,\pi} X\) is determined by the left actions on~\(\GrdH^{(1)}_y\) for all~\(y\), which are in turn determined by the action on~\(\GrdH^{(1)}\).  Thus the functor~\(F\) and the functor induced by the algebraic morphism we have just constructed are equal because they agree on~\(\GrdH^{(1)}\).
\end{proof}

\section{Reconstructing groupoids}

Let~\(\Grd\) be an étale topological groupoid and let \(S\subseteq \Bisec \Grd\) be an inverse sub-semigroup.  Is it sometimes possible to recover~\(\Grd\) from~\(S\)?

We begin with the following well-known result (see~\cite{Exel:Inverse_combinatorial} and also \cite{Matsnev-Resende:Etale_Groupoids}):

\begin{proposition}
  \label{pro:covering_subsemigroup_action}
  Let~\(\Grd\) be an étale topological groupoid and let~\(S\) be an inverse sub-semigroup of \(\Bisec \Grd\).  Assume that~\(S\) covers~\(\Grd^{(1)}\), that is, \(\bigcup S = \Grd^{(1)}\), and that
  \begin{equation}
    \label{eq:basis_condition}
    \text{for all \(s,t\in S\) and \(g\in s\cap t\), there is \(r\in S\) with \(g\in r\subseteq s\cap t\).}
  \end{equation}
  Then~\(\Grd\) is isomorphic to the topological groupoid of germs of the action of~\(S\) on~\(\Grd^{(0)}\).
\end{proposition}

That is, we may recover~\(\Grd\) if we know the object space~\(\Grd^{(0)}\) with the action of~\(S\), provided~\(S\) covers~\(\Grd\) and condition~\eqref{eq:basis_condition} holds.  These conditions together are equivalent to \(S\)~forming a sub-basis for some topology on~\(\Grd^{(1)}\) (not necessarily equivalent to the given topology).  In particular, Proposition~\ref{pro:covering_subsemigroup_action} applies if~\(S\) is a basis for the usual topology of~\(\Grd^{(1)}\).  It is shown in~\cite{Matsnev-Resende:Etale_Groupoids} that if~\(S\) is just an inverse sub-semigroup of~\(\Bisec \Grd\) for which \(E(S)\)~covers~\(\Grd^{(0)}\) (meaning that the inclusion map \(S\hookrightarrow \Bisec \Grd\) is a wide representation of~\(S\) in the sense of \cite{Matsnev-Resende:Etale_Groupoids}*{Definition 2.18}), then the germ groupoid construction yields an étale groupoid with unit space homeomorphic to~\(\Grd^{(0)}\).  This groupoid will, however, not be isomorphic to~\(\Grd\) in general if~\(S\) does not cover~\(\Grd^{(1)}\) (for instance take~\(S\) to be \(E(\Bisec\Grd)\cong \Open(\Grd^{(0)})\) so that \(S\)~is a basis for~\(\Grd^{(0)}\) but the associated groupoid of germs is just the space \(X=\Grd^{(0)}\) viewed as a groupoid in the trivial way).

\begin{lemma}
  \label{lem:recover_space_if_basis}
  Let~\(S\) be an inverse semigroup of bisections of~\(\Grd\) that is a basis for~\(\Grd^{(1)}\).  Assume that \(X\defeq \Grd^{(0)}\) is~\(T_0\).  Let~\(E\) be the idempotent part of~\(S\).  The map \(\delta\colon X\to \widehat{\Open(X)} \to \hat{E}\), \(\delta(x)(e) = [x\in e]\), is an \(S\)\nb-equivariant homeomorphism onto its image.  The groupoid~\(\Grd\) is isomorphic to the restriction of \(\GrdS(S)\) to the invariant subspace \(X\subseteq \hat{E} = \GrdS(S)^{(0)}\).
\end{lemma}

\begin{proof}
  Since~\(E\) is a basis, \(\delta(x)=\delta(y)\) implies that \(x\) and~\(y\) belong to the same open subsets.  Hence \(x=y\) because~\(X\) is~\(T_0\).  Thus~\(\delta\) is injective.  We have \(\delta(x)\in U_e\) if and only if \(e\in E\).  Since the sets \(U_e\) and~\(e\) for \(e\in E\) form bases of~\(X\) and~\(\hat{E}\), the map~\(\delta\) is a homeomorphism onto its image.

  The germ groupoid construction is compatible with restriction to invariant subspaces.  Hence the restriction of \(\GrdS(S)^{(0)}\) to~\(X\) is the groupoid of germs of the action of~\(S\) on~\(X\), which is~\(\Grd\) by Proposition~\ref{pro:covering_subsemigroup_action}.
\end{proof}

It is easy to see that~\(S\) is a basis of~\(\Grd^{(1)}\) if and only if~\(S\) covers~\(\Grd^{(1)}\) and~\(E\) is a basis of~\(\Grd^{(0)}\).  If~\(E\) is not a basis for the topology of~\(\Grd^{(0)}\), then we may equip \(\Grd^{(0)}\) and~\(\Grd^{(1)}\) with the topologies generated by \(E\) and~\(S\), respectively.  This yields a new étale groupoid that may not be distinguished from~\(\Grd\) using our data.  Hence the assumption that~\(E\) is a basis is necessary.

If we do not know the set~\(X\), then we cannot in general recover~\(\Grd\) even if~\(S\) is a basis for~\(\Grd\).  Recall that the open subsets~\(U_e\) for \(e\in E\) form a basis for the topology on \(\GrdS(S)\), so that~\(\Grd\) could be~\(\GrdS(S)\).  But there may be many other étale groupoids~\(\Grd\) for which~\(S\) is a basis of bisections.

The situation improves if we are given the whole inverse semigroup \(\Bisec\Grd\) and know this fact.  The reason for this is that we may recover a sober space from the semilattice \(\Open(X)\) (see Remark~\ref{rem:sober}).

\begin{proposition}
  \label{pro:recover_Grd_from_Bisec}
  Let~\(\Grd\) be a sober étale topological groupoid and let \(S\defeq \Bisec \Grd\).  Let \(E\subseteq S\) be the idempotent part of~\(S\).  Call \(e\in E\) \emph{irreducible} if \(e\neq 1\) and \(e=e_1\cdot e_2\) implies \(e_1=e\) or \(e_2=e\).  For an irreducible \(e\in E\), define \(\varphi_e\colon E\to\{0,1\}\) by \(\varphi_e(f) = 0\) if and only if \(f\le e\).  These maps are characters, and the subset \(X\subseteq\hat{E}\) of all~\(\varphi_e\) with irreducible \(e\in E\) is \(S\)\nb-invariant.  The restriction of \(\GrdS(S)\) to~\(X\) is naturally isomorphic to~\(\Grd\).
\end{proposition}

\begin{proof}
  Recall that \(E=\Open(\Grd^{(0)})\).  An element \(e\in E\) is irreducible in the above sense if and only if its complement in~\(\Grd^{(0)}\) is irreducible as a closed subset.  Since~\(\Grd^{(0)}\) is sober, \(e\in E\) is irreducible if and only if \(e=\Grd^{(0)}\setminus \cl{\{x\}}\) for a unique \(x\in X\).  Then \(\varphi_e(f) = [x\in f]\).  Thus~\(X\) is the range of the canonical map \(\delta\colon \Grd^{(0)} \to \hat{E}\).  The assertion now follows from Lemma~\ref{lem:recover_space_if_basis}.
\end{proof}

\end{document}